\documentclass[12pt]{amsart}
\usepackage{amsmath,amssymb,amsfonts,amsthm,amsopn,bbm}
\usepackage{graphicx,tikz,forest}
\usepackage{pgfplots}
\usepackage[all]{xy}
\usepackage{amsmath}
\usepackage{multirow, longtable, makecell, array}
\usepackage[shortlabels]{enumitem}
\usepackage{amssymb}
\usepackage{mathtools}
\usepackage{bookmark}
\usepackage{hyperref}

\calclayout

\newtheorem{theorem}{Theorem}[section]
\newtheorem{lemma}[theorem]{Lemma}

\newtheorem{proposition}[theorem]{Proposition}
\newtheorem{definition}[theorem]{Definition}

\allowdisplaybreaks 

\newtheorem{example}[theorem]{Example}
\newtheorem{remark}[theorem]{Remark}

\theoremstyle{definition}
\usepackage{microtype}

\usepackage{cellspace} %
\usepackage{caption}
\usepackage{subcaption}

\def\bR{\mathbb{R}}
\def\bC{\mathbb{C}}

\def\bZ{\mathbb{Z}}

\def\cI{\mathcal{I}}

\def\rd{\bR^d}

\def\zd{\bZ^d}

\def\la{\langle}
\def\ra{\rangle}
\def\lc{\left(}
\def\rc{\right)}

\def\supp{\mathrm{supp}}

\def\*b{*_{\bullet}}

\def\Bd'{B_{\delta'}}

\def\cBd'{\bar{B}_{\delta'}}

\def\ird{\int_{\rd}}

\DeclareMathOperator*{\esssup}{ess\,sup}

\def\loc{\mathrm{loc}}

\def\xpqr{X^{p,q}_r}
\def\xpq{X^{p,q}}
\def\xopt{X^{\infty,2}_r}
\def\Ftw{F_{\tau,\omega}}
\def\Fto{F_{\tau,0}}
\def\Fow{F_{0,\omega}}

\newcommand{\adm}[1]{{\left\vert\kern-0.25ex\left\vert\kern-0.25ex\left\vert #1 
		\right\vert\kern-0.25ex\right\vert\kern-0.25ex\right\vert}}

\renewcommand{\adm}[1]{{\vert\kern-0.25ex\vert\kern-0.25ex\vert #1 
		\vert\kern-0.25ex\vert\kern-0.25ex\vert}} %AGGIUNTA FABIO

\newcommand\numberthis{\addtocounter{equation}{1}\tag{\theequation}}

\newtheorem{asu}{Assumption}

\numberwithin{equation}{section}
\makeatletter
\newcommand{\pushright}[1]{\ifmeasuring@#1\else\omit\hfill$\displaystyle#1$\fi\ignorespaces}
\newcommand{\pushleft}[1]{\ifmeasuring@#1\else\omit$\displaystyle#1$\hfill\fi\ignorespaces}
\makeatother
\begin{document}
	
\title[Generalized moduli of continuity under deformations]{Generalized moduli of continuity under irregular or random deformations via multiscale analysis}
\author{Fabio Nicola}
\address{Dipartimento di Scienze Matematiche, Politecnico di Torino, Corso Duca degli Abruzzi 24, 10129 Torino, Italy.}
\email{fabio.nicola@polito.it}
\author{S. Ivan Trapasso}
\address{Dipartimento di Scienze Matematiche, Politecnico di Torino, Corso Duca degli Abruzzi 24, 10129 Torino, Italy.}
\email{salvatore.trapasso@polito.it}
%\date{}
% \date{\today} 

\begin{abstract}
Motivated by the problem of robustness to deformations of the input for deep convolutional neural networks, we identify signal classes which are inherently stable to irregular deformations induced by distortion fields $\tau\in L^\infty(\rd;\rd)$, to be characterized in terms of a generalized modulus of continuity associated with the deformation operator. 

Resorting to ideas of harmonic and multiscale analysis, we prove that for signals in multiresolution approximation spaces $U_s$ at scale $s$, stability in $L^2$ holds in the regime $\|\tau\|_{L^\infty}/s\ll 1$ --- essentially as an effect of the uncertainty principle. Instability occurs when $\|\tau\|_{L^\infty}/s\gg 1$, and we provide a sharp upper bound for the asymptotic growth rate. The stability results are then extended to signals in the Besov space $B^{d/2}_{2,1}$ tailored to the given multiresolution approximation. We also consider the case of more general time-frequency deformations. 

Finally, we provide stochastic versions of the aforementioned results, namely we study the issue of stability in mean when $\tau(x)$ is modeled as a random field (not bounded, in general) with identically distributed variables $|\tau(x)|$, $x\in\rd$.  
\end{abstract}
\maketitle

\section{Introduction}
\subsection{The problem of stability to deformations}
In this note we consider a mathematical problem motivated by the theory and practice of machine learning, that is the robustness of the output of a neural network under modifications of the input datum. Let us briefly illustrate this issue by considering a function $f\colon \rd \to \bR$. Some basic transformations to be taken into account involve \textit{intensity perturbations}, that is $\tilde{f}(x)=f(x)+h(x)$ for some $h \colon \rd \to \bR$, or \textit{signal deformations}, namely $\tilde{f}(x)= F_\tau f(x) \coloneqq f(x-\tau(x))$ for some distortion field $\tau\colon \rd \to \rd$. We stress that this model encompasses natural transformations such as translations or rotations. 

Regardless of the variety of the architectures, the network under our attention can be represented by a map $\Phi$ from $L^2(\rd)$ to some Banach space with norm $\adm{\cdot}$. In order to better appreciate the relevant phenomena, let us consider the classification setting where $\Phi$ acts as a feature extractor. A fair degree of stability of $\Phi$ to small transformations of the input signal is a naturally desirable property in several contexts. For example, consider the classic learning task of digit recognition from images of handwritten symbols, where the input signals suffer from both intra-class and inter-class variance, due for instance to differences in the position of the digit with respect to the background or handwriting styles. As a rule of thumb, it is expected that a small distortion of $f$ into $\tilde{f}$ should correspond to small norm discrepancy $\adm{\Phi(\tilde{f})-\Phi(f)}$ at the level of features. 

The previous remarks thus lead us to require that $\Phi$ enjoys a Lipschitz regularity condition:
\begin{equation}\label{eq intro fe lip} \adm{\Phi(\tilde{f})-\Phi(f)} \le C\| \tilde{f}-f\|_{L^2}, \quad f,\tilde{f} \in L^2(\rd). 
\end{equation}
The smallest constant $C>0$ for which such an estimate holds will be denoted by ${\rm Lip}(\Phi)$. Moreover, in the particular case of a deformation $\tilde{f}=F_\tau f$ of $f$, it would be desirable for $\adm{\Phi(F_\tau f)-\Phi(f)}$ to be small whenever $\tau$ is small with respect to some distortion metric. We can distinguish at least two different angles on the matter:
\begin{itemize}
    \item In keeping with the spirit of geometric deep learning \cite{bro_gdl}, \textit{structural stability} guarantees are inferred from global and local invariance requirements that are \textit{a priori} embedded in the design of the network. A prominent example in this connection is provided by the analysis of the scattering transform introduced in \cite{mall cpam} (see also \cite{brunamallat}: if $\Phi$ is a scattering transform with fixed wavelets filters, modulus nonlinearity and no pooling stages, it was proved in \cite[Proposition 2.5]{mall cpam} that $\Phi$ is a non-expansive transform (i.e., $\rm{Lip}(\Phi)=1$), and in \cite[Theorem 2.12]{mall cpam} that, for every $\tau\in C^2(\rd;\rd)$ with $\|\nabla \tau\|_{L^\infty}\leq 1/2$, 
\begin{equation}\label{eq mallat0}
\adm{\Phi(F_\tau f)-\Phi(f)}\leq C (2^{-J}\|\tau\|_{L^\infty} +\max\{J,1\}\|\nabla \tau\|_{L^\infty}+\|H \tau\|_{L^\infty})\|f\|_{\mathrm{scatt}},
\end{equation}
where $\|f\|_{\mathrm{scatt}}$ is a mixed $\ell^1(L^2)$ scattering norm (which is finite for functions with a logarithmic Sobolev-type regularity), $H \tau$ denotes the Hessian of $\tau$  and $2^J$ is the coarsest scale in the dyadic multiscale analysis associated with the network filters.

    \item In the case where little information on the architecture of the network is available or exploitable, one can only assume to satisfy a Lipschitz condition as in \eqref{eq intro fe lip}. In such cases, stability results for $\Phi$ can be possibly \textit{inherited} from the inherent robustness to deformations of certain input signal classes. This amount to determine a subset $\mathcal{E}\subset L^2(\rd)$ such that bounds for $\|F_\tau f-f\|_{L^2}$ in terms of some complexity metric of $\tau$ can be proved if $f \in \mathcal{E}$. This is the essence of the \textit{decoupling method} introduced in \cite{wiat paper,wiat old,grohs wiat} to obtain stability results for generalized scattering networks by exploiting \textit{sensitivity estimates} of the form $\|F_\tau f-f\|_{L^2} \le C_{\mathcal{E}}\|\tau\|_{L^{\infty}}^{\alpha_{\mathcal{E}}} \|f\|_{L^2}$, which are proved for several classes of interest (including Lipschitz, band-limited and cartoon functions) and deformations $\tau\in C^1(\rd;\rd)$ with $\|\nabla \tau\|_{L^\infty}$ sufficiently small\footnote{Precisely, $\|\nabla \tau \|_{L^\infty}\leq 1/2d$ in \cite{wiat paper} and $\| \nabla \tau \|_{L^\infty}\leq 1/2$ in \cite{mall cpam}. This discrepancy is due to the definition $\| \nabla \tau \|_{L^\infty}\coloneqq \| |\nabla \tau| \|_{L^\infty}$  where $|\nabla \tau|$ is the Frobenius norm of the matrix $\nabla \tau(x)$  in \cite{mall cpam} and the $\ell^\infty$ norm of its entries in \cite{wiat paper}.}.
\end{itemize}

A detailed comparison between Mallat's scattering transform and generalized scattering networks would lead us too far. For our purposes, we just stress that in both cases the results are proved for regular (i.e., at least $C^1$) deformations. In the case of the scattering transform, one is ultimately confronted with the interplay between the network multiscale architecture and the deformation regularity. Consider the case where $f$ is a band-pass function; roughly speaking, the condition $\|\nabla \tau\|_{L^\infty}\leq 1/2$ guarantees that $F_\tau f$ is still localized in frequency, essentially in the same band of $f$, therefore a stability result as in \eqref{eq mallat0} is reasonable (although highly non-trivial to prove) since the network separates scales by design. 

On the other hand, the scope of the decoupling method goes beyond the analysis of generalized scattering transforms: the weak requirement that $\Phi$ is Lipschitz stable as in \eqref{eq intro fe lip} allows us to virtually encompass any neural network where detailed information on structural stability is merely not available. Actually, while most of real-life neural networks are empirically observed to enjoy Lipschitz stability \cite{scaman}, assuming solely this condition about the feature extractor is a worst-case scenario, since other elusive forms of regularity are heuristically expected to occur as well --- such as regularization and cancellation phenomena across hidden layers. In fact, the mathematical literature in this respect is quite limited (see e.g., \cite{balan,zou}) and the available provable bounds for ${\rm Lip}(\Phi)$ are usually quite pessimistic, as they do not exploit further structural information on the network. 

Let us also highlight that, as observed in \cite{mall cpam}, the condition $\|\nabla \tau\|_{L^\infty}\leq 1/2$ can be relaxed to $\|\nabla \tau\|_{L^\infty}<1$ but then the constant blows up when $\|\nabla \tau\|_{L^\infty}\to 1$. The same remark applies to the constants $C_{\mathcal{E}}$ of sensitivity bounds proved in \cite{wiat paper} for band-limited functions and in \cite{koller} for functions in the Sobolev space\footnote{Actually, the result in \cite{koller} is stated for functions in the Sobolev space $H^2(\rd)$. Inspection of the proof and an easy density argument show that it actually holds for functions in the Sobolev space $H^1(\rd)$ of functions $f\in L^2(\rd)$ such that $\|\nabla f\|_{L^2}<\infty$.} $H^1(\rd)$.  It is thus natural to wonder whether stability results can be derived if $\|\nabla \tau\|_{L^\infty}\geq 1$ (therefore $x\mapsto x-\tau(x)$ is no longer invertible) or even for less regular deformations, such as discontinuous ones. Broadly speaking, irregular perturbations such as local pixel shuffling of an image proved to be involved in sophisticate adversarial models such as pixel deflection \cite{pixdef}. They could also be used to model local distortion errors arising in signal encoding, where robustness of classification is naturally expected, as well as to compare contiguous frames of a video where pixels locally move in an irregular fashion (i.e., discontinuous optical flows, pose estimation). 

\subsection{Robustness to irregular deformations}

The previous discussion suggests that the interplay between the deformation regularity and the network structure is a subtle issue. In fact, it turns out that, unless a network is purposefully designed to be stable to irregular deformations, stability results for $\Phi$ at this low-regularity level can only be obtained via the decoupling methods, hence passing on the robustness issue to the input signal class. Indeed, in the context of irregular deformations,  even for well structured networks such as the wavelet scattering ones, it may happen that $\adm{\Phi(F_\tau f)-\Phi(f)}\approx \|F_\tau f-f\|_{L^2}$. 

To be more precise, let us illustrate two kinds of peculiar phenomena that could occur when dealing with irregular deformations --- see also \cite{NT_stab} for further details.  

\begin{enumerate}[(a)]
\item Consider a band-pass function $f$ oscillating at frequency $1/s$ ($s>0$ being the scale); even if $\|\tau\|_{L^\infty}$ is small, it may very well happen that the energy of $F_\tau f$ is amplified by a factor $(\|\tau\|_{L^\infty}/s)^{d/2}$; see Figure \ref{figura intro}. Hence, if $\Phi$ is any energy preserving map ($\|f\|_{L^2}\lesssim \adm{\Phi(f)}\lesssim\|f\|_{L^2}$) then it follows from the triangle inequality that $\adm{\Phi(F_\tau f)-\Phi(f)}/\|f\|_{L^2}\gtrsim (\|\tau\|_{L^\infty}/s)^{d/2}$ when $\|\tau\|_{L^\infty}$ is large \textit{compared to}  $s$. 

\item Let $f$ be a band-pass function, as above, oscillating at frequency $1/s$; even if $\|\tau\|_{L^\infty}$ is small, when $\|\tau\|_{L^\infty}$ is comparable to $s$ it may happen that $f$ and $F_\tau f$ are localized in different dyadic frequency bands, see Figure \ref{figura intro due}. In particular, if $\Phi$ is a wavelet scattering network, their energy will propagate along separate frequency paths and thus the error $\adm{\Phi(F_\tau f)-\Phi(f)}^2 \approx \adm{\Phi(F_\tau f)}^2+\adm{\Phi(f)}^2$ will not be small if $\Phi$ is energy preserving. \end{enumerate}

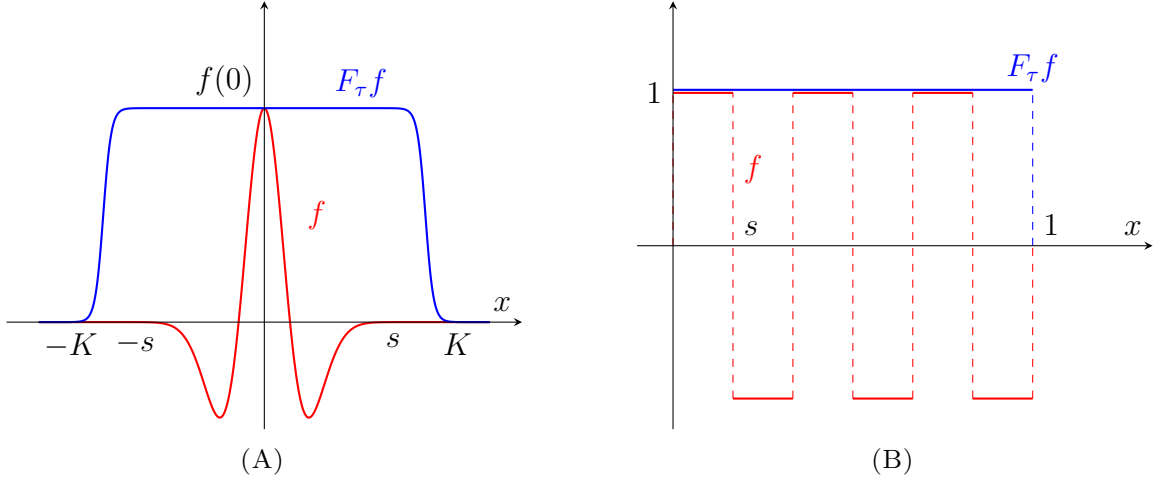
\begin{figure}[h!]
\centering
\captionsetup{width=0.95\linewidth}
     \begin{subfigure}[b]{0.45\linewidth}
     \begin{tikzpicture} \begin{axis}[samples=499,xmin=-4, xmax=4,
      ytick=\empty, 
      xtick=\empty,
    ymin=-0.5, ymax=1.5,
    axis lines=center,
    axis on top=true,
    domain=-3:3,
    xlabel=$x$,]
\addplot [red, thick, domain=-3.5:3.5] plot (\x, {(1-(2*pi*(\x)*(\x)))* e^(-pi*(\x)*(\x))});
\addplot [blue,thick,domain=-3.5:0]    plot (\x, {0.5*(1+tanh(8*(\x+2.5)))});
\addplot [blue,thick,domain=0:3.5]    plot (\x, {0.5*(1+tanh(-8*(\x-2.5)))});
\draw (axis cs:3,0) node[below]{$K$} ;
\draw (axis cs:-3,0) node[below]{$-K$} ;
\draw (axis cs:2,0) node[below]{$s$} ;
\draw (axis cs:-2,0) node[below]{$-s$} ;

\draw (axis cs:0,1) node[above left]{$f(0)$} ;
\draw (axis cs:0.5,0.5) node[right,red]{$f$} ;
\draw (axis cs:1.5,1) node[above,blue]{$F_\tau f$} ;
\end{axis}
\end{tikzpicture}
\caption{}
\label{figura intro}
\end{subfigure}
     \hfill
     \begin{subfigure}[b]{0.45\linewidth}
     \begin{tikzpicture} \begin{axis}[samples=499,xmin=-0.3, xmax=4,
      ytick=\empty, 
      xtick=\empty,
    ymin=-0.6, ymax=0.8,
    axis lines=center,
    axis on top=true,
    domain=-3:3,
    xlabel=$x$,]
\addplot [jump mark left,no marks,blue,thick] coordinates {(0,0.51) (3,0.51)};
\addplot [jump mark left,no marks,red,thick] coordinates {(0,0.5) (0.5,-0.5) (1,0.5) (1.5,-0.5) (2,0.5) (2.5,-0.5) (3,0.5)};
\draw [dashed, blue] (axis cs:3,0) -- (axis cs:3,0.51);

\draw [dashed, red] (axis cs:0,0) -- (axis cs:0,0.5);
\draw [dashed, red] (axis cs:0.5,-0.5) -- (axis cs:0.5,0.5);
\draw [dashed, red] (axis cs:1,-0.5) -- (axis cs:1,0.5);
\draw [dashed, red] (axis cs:1.5,-0.5) -- (axis cs:1.5,0.5);
\draw [dashed, red] (axis cs:2,-0.5) -- (axis cs:2,0.5);
\draw [dashed, red] (axis cs:2.5,-0.5) -- (axis cs:2.5,0.5);
\draw [dashed, red] (axis cs:3,-0.5) -- (axis cs:3,0);

\draw (axis cs:3,0) node[above right]{$1$} ;
% \draw (axis cs:-3,0) node[below]{$-K$} ;
% \draw (axis cs:2,0) node[below]{$s$} ;
% \draw (axis cs:-2,0) node[below]{$-s$} ;

\draw (axis cs:0,0.5) node[ left]{$1$} ;
% \draw (axis cs:0,-0.5) node[below right]{$-1$} ;

\draw (axis cs:0.5,0.25) node[right,red]{$f$} ;
\draw (axis cs:3,0.5) node[above,blue]{$F_\tau f$} ;
\draw (axis cs:0.5,0) node[above right]{$s$} ;
\end{axis}
\end{tikzpicture}
\caption{}
\label{figura intro due}
\end{subfigure}
\caption{(A) A signal $f$ supported on $[-s,s]$ and its deformation $F_\tau f$, where $\tau(x)=x$ for $|x|<K$, with $K>s$. The plateau level corresponds to the value $f(0)$. The operator $F_\tau$ (with the choice of $\tau$ specified above) performs a single-point sampling of $f$, hence it does not make sense on discontinuous signals.
\\[1ex] (B) A signal $f$ localized in frequency where $|\omega|\approx s^{-1}$. With the choice of the deformation $\tau=s\mathbbm{1}_{\{f=-1\}}$, the signal $F_\tau f$ is low-pass (a similar example with $f$ continuous is easily obtained by smoothing the steps).
}
\end{figure}

These phenomena are evident sources of instability in the case where $\|\tau\|_{L^\infty}/s\gg 1$ and $\|\tau\|_{L^\infty}/s\approx 1$ respectively. In passing, note that in order for $F_\tau f$ to be well defined as an element of $L^2(\rd)$ for every $\tau\in L^\infty(\rd;\rd)$, independently of the representative of $f$ in $L^2(\rd)$, $f$ must be assumed continuous at least --- see again Figure \ref{figura intro} for a concrete reference.

We thus conclude that for irregular deformations one is forced to shift the robustness problem from the network architecture to the signal class. In keeping with the spirit of mathematical analysis, let us emphasize that proving bounds for $\|F_\tau f-f\|_{L^2}$ in terms of the deformation size $\|\tau\|_{L^\infty}$ and $\|f\|_{L^2}$ for suitable signal classes can be thought of as a generalization of a typical problem of harmonic analysis, where the differentiability properties of certain function spaces are quantitatively measured in terms of the magnitude of some $L^p$ modulus of continuity $\omega_p[f](t) \coloneqq \| f(x+t)-f(x) \|_{L^p_x}$, $t \in \rd$, as $|t| \to 0$ --- cf.\ for instance \cite[Chapter V]{stein} for a classic reference on the topic. This approach allows one to fine tune the regularity scale of a signal in a very precise way.

The previous remarks motivate focusing on a family of spaces where a precise tuning of the scale is available, in order to elucidate the relationship with the deformation size. We resort again to ideas and tools of modern harmonic analysis, namely we consider multiresolution approximation spaces $U_s \subset L^2(\rd)$, $s>0$ \cite{mallbook}, with a Riesz basis given by a sequence of functions of the type $\phi_{s,n}(x)\coloneqq s^{-d/2}\phi((x-ns)/s)$, $n \in \zd$, where $\phi$ is a fixed filter satisfying certain mild regularity and decay conditions (cf.\ Assumptions \ref{asu riesz}, \ref{asu rev 1} and \ref{asu rev 2} in Section \ref{sec mult} below). Different choices of $\phi$ result in diverse multiresolution approximations, including band-limited functions and polynomial splines of order $n\geq 1$ --- see the discussion in Example \ref{rem examples} below for more details. In general, the introduction of a fixed resolution scale is also natural as a mathematical model of a concrete signal capture system --- cf.\ the general A/D and D/A conversion schemes in \cite[Section 3.1.3]{mallbook}, and also \cite{bietti0,bietti} for a similar limited-resolution assumption in a discrete setting. The scale $s$ (or rather $s^{-1}$) can also be viewed as a rough measure of the complexity of the input signal, and the previous discussion suggests that the ratio $\|\tau\|_{L^\infty}/s$ should appear in sensitivity bounds rather than just $\|\tau\|_{L^\infty}$, which is also expected in order to have dimensionally consistent estimates.

\subsection{Generalized moduli of continuity for multiresolution spaces} The core of our first result can be presented as follows. Under suitable assumptions on $\phi$ there exists a constant $C>0$ such that, for every $\tau \in L^\infty(\rd;\rd)$, $s>0$, \begin{equation}\label{eq intro main}
\| F_\tau f - f \|_{L^2} \le \begin{cases}
		C(\|\tau\|_{L^\infty}/s) \| f \|_{L^2} & (\|\tau\|_{L^\infty}/s \le 1) \\
		C(\|\tau\|_{L^\infty}/s)^{d/2} \| f \|_{L^2} & (\|\tau\|_{L^\infty}/s \ge 1) 
	\end{cases}, \quad f \in U_s. \end{equation}
Stability guarantees for any Lipschitz network $\Phi$ can thus be inferred by the fact that $\adm{\Phi(F_\tau f) - \Phi(f)} \le {\rm Lip}(\Phi) \|F_\tau f- f\|_{L^2}$. We refer to Theorem \ref{thm def bound Us} for precise statements. The estimate \eqref{eq intro main} for $\|\tau\|_{L^\infty}/s \le 1$ recovers and extends the results proved in \cite{wiat paper} for band-limited functions, now without any regularity assumption on the deformation. In Section \ref{sec sharp pwr} we show the sharpness of the estimate \eqref{eq intro main} in both regimes $\|\tau\|_{L^\infty}/s \gg 1$ and $\|\tau\|_{L^\infty}/s \ll 1$.

In short, whenever we have a Lipschitz bound, we have a stability result in the regime $\|\tau\|_{L^\infty}/s \ll 1$, which can be explained in heuristic terms as one of the manifold forms of the uncertainty principle --- see below for further comments in this connection. Observe also that the rate of instability agrees with that of the previous discussion in \textit{(a)} when small-size oscillations, compared with the size of the deformation (namely, if $\|\tau\|_{L^\infty}/s \gg 1$), are allowed.

Interestingly, for fixed $f$, we have in any case $\adm{\Phi(F_\tau f) - \Phi(f)}=O(\|\tau\|_{L^\infty})$ as $\|\tau\|_{L^\infty}\to 0$, although this asymptotic estimate is not uniform with respect to $s$. In fact, in sharp contrast with \eqref{eq mallat0}, the factor $1/s$ in front of $\|\tau\|_{L^\infty}$ associates with a feature of the input signal (i.e., the resolution of $f$), whereas the invariance resolution $2^{-J}$ in \eqref{eq mallat0} is a fixed quantity that depends on the architecture of the network. However, the example in Figure \ref{figura intro due} above shows that in the framework of irregular deformations, even for a fixed wavelet scattering network, we cannot hope for an estimate whose quality does not deteriorate when $\|\tau\|_{L^\infty}$ becomes comparable to the size of the oscillations of $f$. We thus infer that while the choice of wavelet filters is crucial in \cite{mall cpam} to manufacture a transform that is Lipschitz stable to the action of small diffeomorphisms, robustness under \textit{small and irregular} deformations obeys a more general rule, as already anticipated above. In this connection, we address the reader to the aforementioned paper \cite{NT_stab}, where instability results are proved for wavelet scattering networks and deformations at low regularity levels, namely for distortion fields $\tau \in C^\alpha(\rd;\rd)$ with $0\le \alpha <1$.  

The assumption that the input signal $f$ belongs to $U_s$ could be judged not realistic in practice. Rather, we often deal with signals that can be well approximated in low-complexity spaces. For such signal classes we have again a stability result, which is briefly outlined here in low-dimensional settings for simplicity --- we refer to Theorem \ref{main teo 0} for a general and precise statement. Let
$V_j\coloneqq U_{2^j}$, $j\in\mathbb{Z}$, be a multiresolution analysis of $L^2(\rd)$. There exists a constant $C>0$ such that, for every $\tau\in L^\infty(\rd;\rd)$,
\[
\| F_\tau f - f\|_{L^2} \leq C \|\tau\|_{L^\infty}^{d/2}\|f\|_{\dot{B}^{d/2}_{2,1}}, \qquad d=1,2 ,
\]
for any $f\in L^2(\rd)$ such that $\|f\|_{\dot{B}^{d/2}_{2,1}}<\infty$, where $\dot{B}^{d/2}_{2,1}$ denotes the homogeneous Besov space tailored to the given multiresolution analysis \cite[Section 9.2.3]{mallbook}. This regularity level looks optimal in general --- as already observed, $f$ should be at least continuous, and therefore in $B^{d/2}_{2,1}(\rd)$ if we consider the scale of $L^2$-based Besov spaces as a reference.

In Section \ref{sec freq mod def} we prove estimates in the same spirit for more general time-frequency deformations of the type $\Ftw f(x) = e^{i\omega(x)}f(x-\tau(x))$. Modulation deformations are relevant in case of spectral distortions of input signals. These deformations are approached here in a ``perturbative'' way --- that is, by reducing to the results already proved for the case $\omega \equiv 0$. 

The main technical tools behind our results are the properties of certain spaces $X^{p,q}_r$, tailored to the deformation scale $r>0$. Such function spaces are usually referred to as Wiener amalgam spaces and were introduced by Feichtinger in the '80s \cite{fei83,fei81}. As the name suggests, they are obtained by means of a norm that amalgamates a local summability of $L^p$ type on balls of radius $r$ with an $L^q$ behaviour at infinity. They are of current use in harmonic analysis and PDEs, possibly under slightly different names and forms --- see for instance \cite{dan,tao}.  

In Section \ref{sec xpqr} we collect the main properties of these spaces, while in Section \ref{sec xopt} we focus on the space $\xopt$ of locally bounded functions, uniformly at the scale $r$, with $L^2$ decay. This choice should not be intended as a mere technical workaround: in Proposition \ref{maint xopt} we prove that this class is indeed the optimal choice when dealing with arbitrary bounded deformations, since for functions $f \in \xopt\cap C(\rd)$ we have the clear-cut characterization 
\[ 	\| f \|_{\xopt} = \max \{ \| F_\tau f \|_{L^2} : \tau\in L^\infty(\rd;\rd), \, \|\tau\|_{L^\infty} \le r \}. \]
Moreover, the local control offered by $\xpqr$ can be effectively exploited to prove a crucial embedding, cf.\ Theorem \ref{thm rev hold}, which can be heuristically referred to as a reverse H\"older-type inequality for signals in $U_s$ in the spirit of \cite[Lemma 2.2]{tao}, which can be regarded as a novel form of the already mentioned uncertainty principle. Intuitively, if a function $f$ is localized in a low-frequency ball of radius $R^{-1}$ centered at the origin, then $f$ is approximately constant on balls of radius $R$. As a result, deliberately ignoring the effect of the tails, its $L^\infty$ norm on a ball of radius $r<R$ can be roughly bounded by the $L^2$ norm on the same ball (up to a factor $(R/r)^{d/2}$). Strictly speaking, amalgam spaces are needed to put these heuristic remarks on a rigorous ground, leading precisely to the reverse H\"older-type inequality stated in Theorem \ref{thm rev hold}. 

We adopted so far a deterministic model for the deformation, namely a ball in $L^\infty(\rd;\rd)$, without any additional structure, and therefore we provided stability guarantees in a worst-case scenario. In Section \ref{sec random} we assume instead that $\tau$ is a random field with identically distributed variables $|\tau(x)|$, $x\in\rd$. We accordingly study the issue of stability in mean, providing stochastic versions of the above results. For example, we prove that
\begin{equation} \label{eq stoc intro}
\mathbb{E} \| F_\tau f-f\|_{L^2}^2\leq C \mathbb{E}[|\tau|^d]\|f\|^2_{\dot{B}^{d/2}_{2,1}}, \qquad d=1,2,
\end{equation}
see Theorem \ref{teo random} for the precise statement in any dimension, and for similar results when $f$ belongs to limited-resolution spaces $U_s$ as above. Here we set $\mathbb{E}[|\tau|^d]$ for $\mathbb{E}[|\tau(x)|^d]$, the latter being in fact independent of $x$. We also emphasize that the field $\tau$ is no longer assumed to be bounded.

\section{Notation}\label{sec notation}

{\setlength{\parindent}{0cm} 
\setlength{\parskip}{0.2cm}
The open unit ball of $\rd$ with radius $r>0$ and centered at the origin is denoted by $B_r$. 

We introduce a number of operators acting on $f\colon \rd \to \bC$: 
\begin{itemize}
	\item the dilation $D_\lambda$ by $\lambda \ne 0$: $D_\lambda f(y) = f(\lambda y)$;
	\item the translation $T_x$ by $x \in \rd$: $T_x f(y) = f(y-x)$;
	\item the modulation $M_\xi$ by $\xi \in \rd$: $M_\xi f(y) = e^{i y \cdot \xi} f(y)$;
	\item the reflection: $\cI f(y) = f(-y)$;
	\item the Fourier transform (whenever meaningful, e.g.\ if $f \in L^1(\rd)$), normalized here as
	\[ \widehat{f}(\omega) =\mathcal{F}(f)(\omega)= \ird e^{-i\omega \cdot y} f(y) dy. \] 
\end{itemize}

The space $L^\infty(\rd;\rd)$ contains all the measurable vector fields $\tau\colon \rd \to \rd$ such that 
\[ \| \tau \|_{L^\infty} \coloneqq \esssup_{y\in\rd} |\tau(y)| < \infty. \]

We introduce the inhomogeneous magnitude $\la y \ra$ of $y \in \rd$, that is $\la y \ra \coloneqq (1+|y|^2)^{1/2}$. 

The symbol $\mathbbm{1}_E$ will be used to denote the characteristic function of a set $E$. 

While in the statements of the results we will keep track of absolute constants in the estimates, in the proofs we will heavily make use of the symbol $X \lesssim Y$, meaning that the underlying inequality holds up to a universal positive constant factor, namely 
\[ X \lesssim Y \quad\Longrightarrow\quad\exists\, C>0\,:\,X \le C Y. \] Moreover, $X \asymp Y$ means that $X$ and $Y$ are \textit{equivalent quantities}, that is both $X \lesssim Y$ and $X\lesssim Y$ hold. 

In the rest of the note all the derivatives are to be understood in the distribution sense, unless otherwise noted. }

\section{Multiscale Wiener amalgam spaces}\label{sec xpqr}
The following family of function spaces will play a key role in the following. 
\begin{definition}
	For $1 \le p,q \le \infty$ and $r>0$, we denote by $X^{p,q}_r$ the space of all the complex-valued measurable functions in $\rd$ such that
	\begin{equation}\label{eq X norm}
		\| f \|_{\xpqr} \coloneqq \lc \ird \| T_{-x} f \|_{L^p(B_r)}^q dx \rc^{1/q} < \infty, 
	\end{equation} with obvious modifications if $q=\infty$. In the case where $r=1$ we write $\xpq$ for $X^{p,q}_1$. 
\end{definition}
Let us emphasize that $X^{\infty,1}$ coincides with the well known Wiener space of harmonic analysis (cf.\ e.g.\ \cite[Section 6.1]{grobook}). More generally, $\xpq$ coincides with the Wiener amalgam space $W(L^p,L^q)$ of functions with local regularity of $L^p$ type and global decay of $L^q$ type, first introduced by Feichtinger in the '80s \cite{fei83,fei81}; recall that the latter is a Banach space provided with the norm
\[ \| f \|_{W(L^p,L^q)} = \lc \ird \| T_{-x}f \|_{L^p(Q)}^q dx \rc^{1/p},  \]
where $Q\subset \rd$ is an arbitrary compact set with non-empty interior. In fact, different choices of $Q$ yield equivalent norms; typical choices include $Q=B_1$ and $Q=[0,1]^d$. Moreover, the following equivalent discrete-type norm can be used to measure the amalgamated regularity: 
\begin{equation}\label{eq disc norm}
	\| f \|_{X^{p,q}} \asymp \lc \sum_{k \in \zd} \| T_{-k} f \|_{L^p(Q)}^q \rc^{1/q}, \quad Q=[0,1]^d. 
\end{equation}

We also highlight that $X^{p,p}_r$ coincides with $L^p(\rd)$ as set for any $1 \le p \le \infty$, but the norm is rescaled:
\[ \| f \|_{X^{p,p}_r} = r^{d/p} \|f \|_{L^p}. \]
A similar change-of-scale property holds with respect to $X^{p,q}$, in the sense of the following result. 
\begin{lemma}\label{lem rescal norm}
	For any $1 \le p,q \le \infty$ and $r>0$, we have that $\xpqr = \xpq$ as sets, and
	\[ \| f \|_{\xpqr} = r^{d \lc \frac{1}{p} + \frac{1}{q} \rc} \| D_r f \|_{\xpq}. \]
\end{lemma}
\begin{proof} Let us consider the case $p,q < \infty$ for conciseness, the other cases following easily. A straightforward computation shows that
	\begin{align*}
	\| f \|_{\xpqr} & = \lc \ird \lc \int_{B_r} |f(x+y)|^p dy \rc^{q/p} dx \rc^{1/q} \\  & = r^{d/p} \lc  \ird \lc \int_{B_1} |f(x+rz)|^p dz \rc^{q/p} dx \rc^{1/q}  \\
		& = r^{d/p} \lc \ird \lc \int_{B_1} |D_r f(r^{-1}x+z)|^p dz \rc^{q/p} dx \rc^{1/q} \\
		& = r^{d(1/p+1/q)} \lc \ird \lc \int_{B_1} |D_r f(x+z)|^p dz \rc^{q/p} dx \rc^{1/q},
	\end{align*} that is the claim. 
\end{proof}
For future reference let us examine some properties of the spaces $\xpqr$. First, we prove an embedding result that will be often used below.

\begin{proposition}
	For any $1 \le p_1,p_2,q \le \infty$ with $p_1 \le p_2$, and $r>0$, we have
	\[ \| f \|_{X^{p_1,q}_r} \le C r^{d \lc \frac{1}{p_1}-\frac{1}{p_2} \rc} \| f \|_{X^{p_2,q}_r}, \] where the constant $C>0$ depends only on $d$. 
\end{proposition}
\begin{proof} Fix $x \in \rd$ and consider the mapping $h_x \colon y \mapsto |f(x+y)|$. The standard H\"older inequality on the ball $B_r$  yields, with $\rho$ such that $1/p_1 = 1/p_2 + 1/\rho$,
	\begin{align*}
		\| T_{-x} f \|_{L^{p_1}(B_r)} & = \|h_x \cdot \mathbbm{1}_{B_r} \|_{L^{p_1}(B_r)} \\
		& \le \| T_{-x} f \|_{L^{p_2}(B_r)} \|1_{B_r}\|_{L^\rho} \\
		& \le (C r^d)^{\lc \frac{1}{p_1}-\frac{1}{p_2}\rc } \| T_{-x} f \|_{L^{p_2}(B_r)},
	\end{align*} where $C$ is the volume of the $d$-ball with radius $1$. The claim thus follows. 
\end{proof}

In the following results we illustrate the behaviour of the spaces $\xpqr$ under convolution and dilations. In fact, the case with $r=1$ is covered by the standard theory of amalgam spaces (cf.\ \cite{fei83,heil} and \cite[Proposition 2.2]{cn sharp} respectively), hence the result for $r \ne 1$ follows by rescaling the norms in accordance with Lemma \ref{lem rescal norm}. 
\begin{proposition}\label{prop conv spaces}
	For any $r>0$ and $1 \le p_1,p_2,p,q_1,q_2,q \le \infty$ such that 
	\[ \frac{1}{p_1} + \frac{1}{p_2} = 1 + \frac{1}{p}, \qquad \frac{1}{q_1} + \frac{1}{q_2} = 1 + \frac{1}{q}, \] we have
	\[ \| f*g \|_{\xpqr} \le C r^{-d} \| f \|_{X^{p_1,q_1}_r} \|g \|_{X^{p_2,q_2}_r},  \] for a constant $C>0$ that depends only on $d$.
\end{proposition}

\begin{proposition}\label{prop dil spaces}
	For any $r,s>0$ and $1 \le p,q \le \infty$ we have
\[ \| D_s f \|_{\xpqr} \le \begin{cases} C s^{-d \max(1/p,1/q)} \|f \|_{\xpqr} & (0<s\le 1) \\ C s^{-d \min(1/p,1/q)} \|f \|_{\xpqr}  & (s \ge 1) \end{cases}, \]
for a constant $C>0$ that depends only on $d$. 
\end{proposition}

\section{$L^\infty$ deformations and the space $\xopt$}\label{sec xopt}
Let us consider the  class of deformation mappings $F_\tau$ associated with distortion functions $\tau\colon \rd \to \rd$ by setting
\[ F_\tau f(x) \coloneqq f(x-\tau(x)),
 \]
where $f\colon \rd \to \bC$.

We prove that the class $\xopt$ is the optimal choice as far as sensitivity bounds for arbitrary bounded deformations are concerned. The second part of the following result can be regarded as a linearization of a maximal operator (cf.\ \cite[Section 6.1.3]{grafakos_mod}). 

\begin{proposition}\label{maint xopt}
	We have
	\begin{equation}\label{eq xopt bound}
		 \| F_\tau f \|_{L^2} \le \|f \|_{\xopt}, \quad r=\| \tau\|_{L^\infty},
	\end{equation}
	 for every $f \in \xopt\cap C(\rd)$ and $\tau \in L^\infty(\rd;\rd)$. 
	
	More precisely, for every function $f \in \xopt\cap C(\rd)$, we have the characterization
	\begin{equation}\label{eq xopt char}
		\| f \|_{\xopt} = \max \{ \| F_\tau f \|_{L^2} \, : \, \tau\in L^\infty(\rd;\rd), \, \|\tau\|_{L^\infty} \le r \}.
	\end{equation}
\end{proposition}
\begin{remark}
	Note that the continuity assumption on $f \in \xopt$ is essential in the statement, otherwise $f(x-\tau(x))$ may not even be well defined in $L^2$ (i.e., independent of the representative $f$), as evidenced by the case $\tau(x)=x$ for $x \in B_R$ and small $R>0$. See also Figure \ref{figura intro} in this connection.
\end{remark}
\begin{proof}[Proof of Proposition \ref{maint xopt}]
	It is clear that, for almost every $x\in\rd$, 
	\[ |f(x-\tau(x))| \le \sup \{ |f(x-y)| : y \in \rd,\, |y|\le \|\tau\|_{L^\infty} \}, \] and thus \eqref{eq xopt bound} follows after taking the $L^2$ norm (the above supremum is the same as the essential supremum because $f$ is continuous).   
	
	For what concerns \eqref{eq xopt char}, it is enough to prove that 
	\[ 		\| f \|_{\xopt} \le \max \{ \| F_\tau f \|_{L^2} \, : \, \tau\in L^\infty(\rd;\rd), \, \|\tau\|_{L^\infty} \le r \}. \]
	To this aim, notice that if we could design a measurable correspondence $\tau$ between $x \in \rd$ and a point $y^* =\tau(x)\in \overline{B_r}$ where the function $\overline{B_r} \ni y \mapsto |f(x-y)|$ attains its maximum, then 
	\[ \max_{|y|\le r} |f(x-y)| = |f(x-\tau(x))| = |F_\tau(x)|, \] and the desired conclusion would follow once taking the $L^2$ norm. The existence of such a measurable selector is a consequence of the measurable maximum theorem \cite[Theorem 18.19]{aliprantis} (in fact, an easier argument would give \eqref{eq xopt char} with the supremum in place of the maximum, cf.\ \cite[Section 6.1.3]{grafakos_mod}).
	\end{proof}

The following result provides a sensitivity bound for $L^2$ functions which are locally (i.e., on every compact subset) Lipschitz continuous, uniformly at the deformation scale. It should be compared with the result in \cite{koller}, valid for functions in the Sobolev space $H^1(\rd)$ and deformations $\tau\in C^1(\rd;\rd)$ with $\|\nabla \tau\|_{L^\infty}\leq 1/2$, hence regular.

\begin{proposition}\label{maint grad xopt}
	There exists a constant $C>0$ such that 
\begin{equation}\label{eq maint grad xopt}
	\| F_\tau f - f \|_2 \le C \|\tau\|_{L^\infty} \|\nabla f \|_{\xopt}, \quad r=\|\tau\|_{L^\infty},
\end{equation} for every $\tau \in L^\infty(\rd;\rd)$ and every function $f \in \xopt$ such that $\| \nabla f \|_{\xopt} < \infty$. 
\end{proposition}
Observe that the condition  $\| \nabla f \|_{\xopt} < \infty$ implies that $\nabla f\in L^\infty_{\loc}(\rd)$, and therefore $f$ is locally Lipschitz continuous after possibly being redefined on a set of measure zero (cf.\ \cite[Theorem 4, page 294]{evans}), in particular $f$ is continuous. In the following we will always identify $f$ with its continuous version. Also, we set \begin{equation}
		\| \nabla f\|_{\xopt} \coloneqq \| | \nabla f| \|_{\xopt}.
	\end{equation}

\begin{proof}[Proof of Proposition \ref{maint grad xopt}]
For $x\in\rd$, $r>0$ let $B(x,r)$ be the open ball in $\rd$ of radius $r$ and center $x$.
By the Poincar\'e inequality for a ball\footnote{That is $\|f - \overline{ f}_{x,r}\|_{L^\infty(B(x,r))}\leq C r\|\nabla f\|_{L^\infty(B(x,r))}$ where $\overline{ f}_{x,r}$ is the average of $f$ over $B(x,r)$. Since under our assumption $f$ is continuous in $\rd$, we can replace the $L^\infty$ norm in the left-hand side by the supremum of $|f|$, and then one obtains \eqref{eq stima lipsc} from the triangle inequality (by adding and subtracting $\overline{ f}_{x,r}$).} (cf.\ \cite[Theorem 2, page 291]{evans}) we see that there exists a constant $C>0$ such that, for every $r>0$ and $x\in\rd$,
\begin{equation}\label{eq stima lipsc}
    |f(x-y)-f(x)|\leq Cr\|\nabla f\|_{L^\infty(B(x,r))}, \quad |y|\leq r.
\end{equation}
Setting $y=\tau(x)$, $r=\|\tau\|_{L^\infty}$ and taking the $L^2$ norm lead to the desired conclusion. 
\end{proof}
	
\section{Multiresolution approximation spaces}\label{sec mult}
Fix $\phi \in L^2(\rd)$ and recall \cite{mallbook} that the associated approximation space $U_s$ at scale $s>0$ is defined as follows:
\[ U_s \coloneqq \overline{\text{span} \{ \phi_{s,n} \}_{n \in \zd}}, \qquad \phi_{s,n}(x) \coloneqq s^{-d/2} T_{ns} D_{1/s} \phi(x) = s^{-d/2} \phi \lc \frac{x-ns}{s} \rc. \] 
In the rest of the paper we are going to deal with the following assumptions on $\phi$. 

\begin{asu} \label{asu riesz}
	There exist constants $A, B >0$ such that 
	\begin{equation}\label{eq cond riesz}
		A \le \sum_{k \in \zd} |\hat{\phi}(\omega-2\pi k)|^2 \le B \quad \text{for a.e. } \omega \in \rd. \end{equation} This is \emph{equivalent} to assuming that $\{ \phi_{s,n} \}$ is a Riesz basis for $U_s$ (cf.\ \cite[Theorem 3.4]{mallbook} in the case where $d=1$, while the result for $d>1$ follows by direct extension of the one-dimensional one).
\end{asu}
We further assume one of the following regularity/decay conditions on $\phi$. 

\begin{asu} \label{asu rev 1} \textit{At least one} of the following conditions holds. 
	\begin{enumerate}[(i)]
	\item $\phi$ belongs to the Wiener space: \begin{equation}\label{eq cond rev 1}
		\phi \in X^{\infty,1},
	\end{equation}  
	in particular $\phi$ is locally bounded and has a $L^1$ decay.
	\item There exist $\alpha > 1/2$ and $B'>0$ such that \begin{equation}\label{eq cond rev 2}
		\sum_{k \in \zd} |(v^\alpha \hat{\phi})(\omega - 2\pi k)|^2 \le B' \quad \text{for a.e. } \omega \in [0,2\pi]^d,
	\end{equation} where we introduced the weight function $v(\omega) = \la \omega_1 \ra \cdots \la \omega_d \ra$, $\omega \in \rd$.
	\end{enumerate}
\end{asu}

\begin{asu}\label{asu rev 2}
 	At least one of the conditions \eqref{eq cond rev 1} and \eqref{eq cond rev 2} of Assumption \ref{asu rev 1} is satisfied for all $\partial_j \phi$, $j=1,\ldots,d$, in place of $\phi$. 
\end{asu}

\begin{example}\label{rem examples} This is a convenient stage where to present some examples of functions satisfying the assumptions. Generally speaking, \eqref{eq cond rev 1} is satisfied by any function $\phi \in L^\infty(\rd)$ with compact support, while the same condition on the Fourier side (i.e., $\hat{\phi} \in L^\infty(\rd)$ with compact support) guarantees that \ref{eq cond rev 2} holds. To be more concrete, let us provide some standard examples in dimension $d=1$ --- Assumption \ref{asu riesz} will be satisfied in all cases (cf.\ \cite[Section 3.1.3, pages 69,70]{mallbook}).
	\begin{itemize}
		\item The choice $\phi = \mathbbm{1}_{[0,1]}$, leading to piecewise constant approximations (block sampling), is easily seen to satisfy \eqref{eq cond rev 1} but not \eqref{eq cond rev 2} for any $\alpha>1/2$, nor Assumption \ref{asu rev 2}.
		\item The normalized sinc function $\displaystyle \phi(x) = \frac{\sin(\pi x)}{\pi x}$, corresponding to Shannon approximations (i.e., band-limited functions), satisfies \eqref{eq cond rev 2} for every $\alpha >0$, as well as Assumption \ref{asu rev 2}, but not \eqref{eq cond rev 1}.
		\item The B-spline $\phi$ of degree $n$, obtained by $n+1$ convolutions of $\mathbbm{1}_{[0,1]}$ with itself and centering at $0$ or $1/2$, can be characterized by its Fourier transform: 
		\[ \hat{\phi}(\omega) = \lc \frac{\sin(\omega/2)}{\omega/2} \rc^{n+1} e^{-i\varepsilon \omega/2}, \quad \varepsilon = \begin{cases}
			1 & (n\text{ is even}) \\ 0 & (n\text{ is odd})
		\end{cases}. \] We see that if $n\ge 1$ then both \eqref{eq cond rev 1} and \eqref{eq cond rev 2} are satisfied (for $\alpha < n+1/2$), as well as Assumption \ref{asu rev 2} (the case $n=0$ is covered by the previous case of $\phi = \mathbbm{1}_{[0,1]}$). 
	\end{itemize} 
\end{example}
In Assumption \ref{asu rev 1} we introduced the weight function $v$. Let us now define a companion Sobolev space, for $\alpha\in\mathbb{R}$, $\alpha\geq 0$:
\[
H^\alpha_\otimes(\rd)\coloneqq\{f\in L^2(\rd):\ \|f\|_{H^\alpha_\otimes}\coloneqq\|v^\alpha \hat{f} \|_{L^2}<\infty\}.
\] Roughly speaking, $H^\alpha_\otimes(\rd)$ consists of functions in $L^2(\rd)$ which have at least $\alpha$ (possibly fractional) derivatives in the directions of the axes in $L^2(\rd)$. It is easy to realize that this space contains functions in the usual Sobolev space $H^{d\alpha}(\rd)$ as well as tensor products $\phi_1\otimes \ldots \otimes \phi_d$, with  $\phi_j\in H^\alpha(\mathbb{R})$, $j=1,\ldots, d$.

\begin{proposition}\label{immersioni}
If $\alpha>1/2$ we have the embedding $H^\alpha_\otimes(\rd)\hookrightarrow L^\infty(\rd)\cap C(\rd)$, as well as
\[
H^\alpha_\otimes(\rd) \hookrightarrow X^{\infty,2}.
\]
\end{proposition}

\begin{proof}
The embedding in $L^\infty$ follows at once from the chain of inequalities
\[
\| f \|_{L^\infty} \lesssim \| \hat{f} \|_{L^1} \le \| v^{-\alpha} \|_{L^2} \| \hat{f}v^{\alpha}\|_{L^2} \]
and the fact that $v^{-\alpha} \in L^2(\rd)$ if $\alpha > 1/2$. 
The embedding in $C(\rd)$ is then clear because the space of Schwartz functions is easily seen to be dense in $H^\alpha_\otimes(\rd)$.

Concerning the embedding in $X^{\infty,2}$, let  $g \in C^\infty_c(\rd)$, with $g=1$ on $B_1$. Then 
\[
 \| f \|_{X^{\infty,2}} \le \| \| T_{-x}f \cdot g \|_{L^\infty} \|_{L^2_x} \lesssim \| \| f \cdot T_x g \|_{H^\alpha_\otimes} \|_{L^2_x} \lesssim \| f \|_{H^\alpha_\otimes}, 
  \] 
  where the last inequality is proved in \cite[Propositon 11.3.1(c)]{grobook}.
\end{proof}

We now establish a crucial reverse H\"older-type inequality for functions in $U_s$.

\begin{theorem}\label{thm rev hold} Let $\phi \in L^2(\rd)$ be such that Assumption \ref{asu riesz} is satisfied. 
	\begin{enumerate}[(i)]
		\item If Assumption \ref{asu rev 1} holds then there exists $C>0$ such that, for every $r,s >0$, 
\begin{equation}\label{eq rev hold}
	\| f \|_{\xopt} \le C (1+ r/s)^{d/2} \| f\|_{L^2}, \quad f \in U_s.
\end{equation} 
\item If Assumption \ref{asu rev 2} holds then there exists $C>0$ such that, for every $r,s >0$,
\begin{equation}\label{eq grad rev hold}
	\| \nabla f \|_{\xopt} \le C s^{-1}(1+r/s)^{d/2} \| f \|_{L^2}, \quad f \in U_s.  
\end{equation}
\end{enumerate}
\end{theorem}

\begin{remark}\label{rem proj}
	Let $P_{U_s}$ be the orthogonal projection operator on $U_s$. Since $\| P_{U_s} \|_{L^2 \to L^2} = 1$, \eqref{eq rev hold} is equivalent to
	\[ \| P_{U_s}f \|_{\xopt} \le C (1+r/s)^{d/2} \| f\|_{L^2}, \quad f \in L^2(\rd). \]
\end{remark}
\begin{proof}[Proof of Theorem \ref{thm rev hold}]
	Let us commence with the proof of \eqref{eq rev hold}. Let $\{ \tilde{\phi}_{s,n} \}_{n \in \zd}$ be the dual basis to $\{ {\phi}_{s,n} \}_{n \in \zd}$. If $f\in U_s$ then 
	\[ f = \sum_{n \in \zd} a_n \phi_{s,n}, \quad a_n \coloneqq \la f, \tilde{\phi}_{s,n} \ra, \] and by Lemma \ref{lem rescal norm} we have
	\begin{align*}
		\| f \|_{\xopt} & = r^{d/2} \| D_r f \|_{X^{\infty,2}} \\ 
		& = \left\| \sum_{n \in \zd} a_n \phi_{s/r,n}  \right\|_{X^{\infty,2}} \\
		& = \lc \frac{r}{s} \rc^{d/2} \left\| \sum_{n \in \zd} a_n \phi \lc \frac{r}{s} \cdot \, - n \rc \right\|_{X^{\infty,2}} \\ 
		& = \lc \frac{r}{s} \rc^{d/2} \left\| \left\| \sum_{n \in \zd} a_n \phi \lc \frac{r}{s} (x+y) - n  \rc \mathbbm{1}_{B_1}(y)  \right\|_{L^\infty_y} \right\|_{L^2_x} \\
		& = \left\| \left\| \sum_{n \in \zd} a_n \phi (x+y - n) \mathbbm{1}_{B_{r/s}} (y) \right\|_{L^\infty_y} \right\|_{L^2_x} \\ 
		& = \left\| \sum_{n \in \zd} a_n T_n \phi  \right\|_{X^{\infty,2}_{r/s}} \\ 
		& \lesssim \lc 1+ \frac{r}{s} \rc^{d/2} \left\| \sum_{n \in \zd} a_n T_n \phi  \right\|_{X^{\infty,2}}, \numberthis \label{eqal f sum}
	\end{align*} where in the last step we used Lemma \ref{lem rescal norm} and Proposition \ref{prop dil spaces}. 

Assume now \eqref{eq cond rev 1}, namely $\phi \in X^{\infty,1}$. Then the conclusion follows from \eqref{eqal f sum} using the equivalent discrete-type norm in \eqref{eq disc norm} (with $Q=[0,1]^d$):
\begin{align*}
	\| f \|_{\xopt} & \lesssim \lc 1+ \frac{r}{s} \rc^{d/2} \left\| \sum_{n \in \zd} a_n T_n \phi  \right\|_{X^{\infty,2}} \\ 
	 & \lesssim \lc 1+ \frac{r}{s} \rc^{d/2} \left\| \sum_{n \in \zd} |a_n| \| \phi (k+y-n) \mathbbm{1}_Q \|_{L^\infty_y} \right\|_{\ell^2_k} \\
	 & \lesssim \lc 1+ \frac{r}{s} \rc^{d/2} \lc \sum_{k \in \zd}  \| \phi (k+y) \mathbbm{1}_Q \|_{L^\infty_y} \rc \lc \sum_{n \in \zd} |a_n|^2 \rc^{1/2} \\ 
	 & \lesssim \lc 1+ \frac{r}{s} \rc^{d/2}\| \phi \|_{X^{\infty,1}} \| f \|_{L^2},
\end{align*} where we used that $\ell^1 * \ell^2 \hookrightarrow \ell^2$ and  $\|  a_n \|_{\ell^2} \lesssim \| f\|_{L^2}$. 

Let us assume \eqref{eq cond rev 2} instead. By \eqref{eqal f sum}, it is enough to show that
\[ \left\| \sum_{n \in \zd} a_n T_n \phi  \right\|_{X^{\infty,2}} \lesssim  \| f \|_{L^2}. 
\] 
Using the embedding in Proposition \ref{immersioni} we obtain
\begin{align*}
	\left\| \sum_{n \in \zd} a_n T_n \phi  \right\|_{X^{\infty,2}} & \lesssim \left\| \sum_{n \in \zd} a_n T_n \phi  \right\|_{H^\alpha_\otimes} \\ 
	& = \left\| \sum_{n \in \zd} a_n e^{-in\omega} \hat{\phi}(\omega) v^\alpha(\omega) \right\|_{L^2} \\
	&  = \lc \int_{[0,2\pi]^d} \sum_{k \in \zd} | F(\omega) (v^\alpha \hat{\phi})(\omega-2\pi k)|^2 d\omega \rc^{1/2} \\
	& \le \| F \|_{L^2([0,2\pi]^d)}   \lc \esssup_{\omega \in [0,2\pi]} \sum_{k \in \zd} | (v^\alpha \hat{\phi})(\omega-2\pi k)|^2 \rc^{1/2} \\
	& \lesssim \| f \|_{L^2}, 
\end{align*} where we set $F(\omega) \coloneqq \sum_{n \in \zd} a_n e^{-in\omega}$ (which is a $2\pi$-periodic, square integrable on $[0,2\pi]$, function), and then used \eqref{eq cond rev 2} and 
\[ \| F \|_{L^2([0,2\pi]^d)}^2 \asymp \sum_{n \in \zd} |a_n|^2 \asymp \| f \|_{L^2}^2. \]

The proof of \eqref{eq grad rev hold} goes along the same lines after differentiation in the representation $f = \sum_{n \in \zd} a_n \phi_{s,n}$; the details are left to the interested reader. 
\end{proof}

\begin{remark}\label{rem cont Us}
	\begin{enumerate}[(i)]
		\item It is easy to realize that if $\phi$ satisfies \eqref{eq cond rev 2} then  $\phi\in H^\alpha_\otimes(\rd)$ (it is enough to integrate both sides of \eqref{eq cond rev 2} on $[0,2\pi]$); as a result, if $\alpha>1/2$ then $\phi$ is continuous by Proposition \ref{immersioni}. 		
		\item If $\phi\in L^2(\rd)$ satisfies Assumption \ref{asu rev 2} then $\phi$ has first order partial derivatives locally in $L^\infty$, hence $\phi$ is locally Lipschitz, therefore continuous.  
		\item If $\phi\in L^2(\rd)$ satisfies the assumption $A$ and $B$ and is continuous, then $U_s \hookrightarrow C(\rd)$ since the truncated sums $\sum_{|n|\le N} a_n T_{sn}\phi$ are continuous and \eqref{eq rev hold} shows that convergence in $L^2$ implies convergence in $\xopt \hookrightarrow L^\infty(\rd)$ for functions in $U_s$.
		\item If $s \ll r$ then the occurrence of the factor $r/s$ in \eqref{eq rev hold} can be heuristically explained by the presence of highly oscillating functions in $U_s$, which are not stable under deformations of ``size'' $r$. 
	\end{enumerate}
\end{remark}
	
We are ready to provide deformation sensitivity bounds for functions in $U_s$. 

\begin{theorem}\label{thm bound Us} 
	Let $\phi \in L^2(\rd)\cap C(\rd)$ satisfy Assumptions \ref{asu riesz} and \ref{asu rev 1}. There exists a constant $C>0$ such that, for every $\tau \in L^\infty(\rd;\rd)$ and $s>0$
	\begin{equation}\label{eq bound Us}
		\| F_\tau f \|_{L^2} \le C (1+\|\tau\|_{L^\infty}/s)^{d/2}\| f\|_{L^2}, \quad f \in U_s.
	\end{equation}
\end{theorem}
\begin{proof} The desired estimate follows by a straightforward concatenation of Proposition \ref{maint xopt}, since the assumptions on $\phi$ imply that $U_s \hookrightarrow C(\rd)$ (cf.\ Remark \ref{rem cont Us}), and Theorem \ref{thm rev hold} with $r=\| \tau \|_{L^\infty}$.  
\end{proof}

\begin{theorem}\label{thm def bound Us}
	Let $\phi \in L^2(\rd) $ be such that Assumptions \ref{asu riesz}, \ref{asu rev 1} and \ref{asu rev 2} are satisfied. There exists a constant $C>0$ such that
	\begin{equation}\label{eq def bound Us}
		\| F_\tau f - f \|_{L^2}  \, \le \begin{cases}
			 C (\|\tau\|_{L^\infty}/s) \|f \|_{L^2} & (\| \tau \|_{L^\infty}/s \le 1) \\
		C (\|\tau\|_{L^\infty}/s)^{d/2} \|f \|_{L^2} & (\| \tau \|_{L^\infty}/s \ge 1)
		\end{cases},
	\end{equation}
	for every $\tau \in L^\infty(\rd;\rd)$, $s>0$ and $f \in U_s$.
\end{theorem}
\begin{proof}
	Let us consider first the case $\|\tau \|_{L^\infty}/s \le 1$. Combining Proposition \ref{maint grad xopt} with Theorem \ref{thm rev hold} with $r=\|\tau\|_{L^\infty}$ we infer, for $f\in U_s$,
	\begin{align*}
		\| F_\tau f -f \|_{L^2} & \lesssim \|\tau \|_{L^\infty} \| \nabla f \|_{\xopt} \\ 
		& \lesssim (\| \tau \|_{L^\infty}/s) (1+\|\tau\|_{L^\infty}/s)^{d/2} \| f \|_{L^2} \\
		& \lesssim (\| \tau \|_{L^\infty}/s) \| f \|_{L^2},
	\end{align*} that is the claim. 

The case $\|\tau \|_{L^\infty}/s \ge 1$ can be approached via the triangle inequality, that is \linebreak $\| F_\tau f - f \|_{L^2} \le \|F_\tau f \|_{L^2} + \| f \|_{L^2}$, and Theorem \ref{thm bound Us}. 
\end{proof}

\begin{remark} \label{rem examples bis}
More generally, the same result of Theorem \ref{thm def bound Us} holds if $f$ is replaced on the left-hand side by $P_{U_s} f$ for $f \in L^2(\rd)$, cf.\ Remark \ref{rem proj}. Moreover, taking into account the examples in Example \ref{rem examples} we see that Theorem \ref{thm def bound Us} applies when $U_s$ are approximation spaces of polynomial splines of degree $n\geq 1$, as well of band-limited functions --- which can be regarded as splines of infinite order. 
\end{remark}

We conclude this section by extending the above stability bounds to signal classes with minimal regularity. 
In addition to the assumptions of Theorem \ref{thm def bound Us}, we suppose that
 $V_j\coloneqq U_{2^j}$, $j\in\mathbb{Z}$, define a multiresolution approximation of $L^2(\rd)$, so that $V_{j+1}\subset V_j$. Let $W_{j+1}$ be the orthogonal complement of $V_{j+1}$ in $V_j$ and $P_{W_j}$ be the corresponding orthogonal projection; for $s\in\mathbb{R}$, the corresponding homogeneous Besov norm \cite[Section 9.2.3]{mallbook} is given by
 \begin{equation}\label{eq besov}
 \|f\|_{\dot{B}^s_{2,1}}=\sum_{j\in\mathbb{Z}}2^{-js}\|P_{W_j}f\|_{L^2}.
 \end{equation}
 
\begin{theorem}\label{main teo 0} Under the same assumptions of Theorem \ref{thm def bound Us}, suppose in addition that
 $V_j\coloneqq U_{2^j}$, $j\in\mathbb{Z}$, define a multiresolution approximation of $L^2(\rd)$.
 
 There exists $C>0$ such that for every $\tau\in L^\infty(\rd;\rd)$ and $f\in L^2(\rd)$ with $\|f\|_{\dot{B}^{d/2}_{2,1}}<\infty$,
\begin{equation}\label{eq stima dim 1 besov 0}
\|F_\tau f - f\|_{L^2} \leq C (\|\tau\|_{L^\infty}\|f\|_{\dot{B}^{1}_{2,1}} +\|\tau\|_{L^\infty}^{d/2}\|f\|_{\dot{B}^{d/2}_{2,1}}), \qquad d\geq 2,
\end{equation}
and
\begin{equation}\label{eq stima dim 1 besov}
\|F_\tau f - f\|_{L^2} \leq C  \|\tau\|_{L^\infty}^{1/2}\|f\|_{\dot{B}^{1/2}_{2,1}}, \qquad d=1.
\end{equation}
\begin{proof}
We consider the decomposition 
\[
f=\sum_{\|\tau\|_{L^\infty}\leq 2^j}P_{W_j} f+ \sum_{2^j<\|\tau\|_{L^\infty}}P_{W_j}f
\]
and apply \eqref{eq def bound Us} to each term, hence we obtain
\begin{equation}\label{eq stima intermedia} 
\|F_\tau f-f\|_{L^2}\lesssim \|\tau\|_{L^\infty}\sum_{\|\tau\|_{L^\infty}\leq 2^j} 2^{-j}\|P_{W_j} f\|_{L^2}+\|\tau\|_{L^\infty}^{d/2} \sum_{2^j<\|\tau\|_{L^\infty}}2^{-jd/2}\|P_{W_j}f\|_{L^2},
\end{equation}
which implies the desired result if $d\geq 2$. 

For $d=1$ it is sufficient to continue the estimate in \eqref{eq stima intermedia} using
\[
\sum_{\|\tau\|_{L^\infty}\leq 2^j} 2^{-j}\|P_{W_j} f\|_{L^2}\leq \sum_{\|\tau\|_{L^\infty}\leq 2^j} 2^{-j/2}2^{-j/2}\|P_{W_j} f\|_{L^2}\leq \|\tau\|_{L^\infty}^{-1/2}\|f\|_{\dot{B}^{1/2}_{2,1}}.
\]
\end{proof}
\begin{remark}
From the very definition \eqref{eq besov} of the Besov norm, it follows that if $d\geq 2$ and $f\in L^2(\rd)$ with $\|f\|_{\dot{B}^{d/2}_{2,1}}<\infty$ then $\|f\|_{\dot{B}^{1}_{2,1}}<\infty$. 

Also, note that even in dimension 1 we have $\|F_\tau f - f\|_{L^2}=O(\|\tau\|_{L^\infty})$ as $\|\tau\|_{L^\infty}\to 0$ for every fixed $f\in U_s$ and every $s>0$, as a consequence of Theorem \ref{thm def bound Us}. However this asymptotic estimate is not uniform in the ball $\|f\|_{L^2}+\|f\|_{\dot{B}^{1/2}_{2,1}}\leq 1$, and the factor $\|\tau\|^{1/2}_{L^\infty}$ in \eqref{eq stima dim 1 besov} is instead optimal when looking for uniform estimates; see the examples in Section \ref{sec sharp pwr} below. In dimension $d\geq 2$ it follows easily from \eqref{eq stima dim 1 besov 0} that $\|F_\tau f - f\|_{L^2}=O(\|\tau\|_{L^\infty})$ as $\|\tau\|_{L^\infty}\to 0$ uniformly for $f$ in the ball $\|f\|_{L^2}+\|f\|_{\dot{B}^{d/2}_{2,1}}\leq 1$.
\end{remark}

\end{theorem}

\section{Frequency-modulated deformations}\label{sec freq mod def}

In this section we extend some results proved so far to the class of time-frequency deformation mappings $\Ftw$ associated with distortion functions $\tau\in L^\infty(\rd; \rd)$, $\omega\in L^\infty(\rd;\mathbb{R})$ by setting
\[ \Ftw f(x) \coloneqq e^{i\omega(x)}f(x-\tau(x)),
\]
where $f\colon \rd \to \bC$. In case of trivially null distortions we write $\Fow$ and $\Fto$ with obvious meaning. 

While most of the results above can be stated and proved with minor updates for general deformations $\Ftw$, we prefer to offer here a different perspective that allows one to reduce to the results for $F_\tau$ in a straightforward way. Indeed, note that $\Ftw = \Fow\Fto$ and $\Fto$ coincides with the deformation $F_\tau$ considered in the previous sections. Moreover, for every $f \in L^2(\rd)$ we have that $\| \Ftw f \|_{L^2} = \|\Fto f \|_{L^2}$ for arbitrary measurable $\omega$, and
\[ \| \Ftw f - f \|_{L^2} \le \| \Ftw f - \Fto f \|_{L^2} + \| \Fto f - f \|_{L^2}. \] The second addend is already covered, while for the first one we have 
\[ \| \Ftw f - \Fto f \|_{L^2} \le \| e^{i\omega}-1 \|_{L^\infty} \|\Fto f\|_{L^2} \le \| \omega \|_{L^\infty} \| \Fto f \|_{L^2}. \] 
As a result, the bounds in Propositions \ref{maint xopt} and \ref{maint grad xopt} generalize as follow. 
\begin{theorem}\label{thm gen xopt} We have
	\begin{equation}\label{eq gen xopt bound}
		 \|\Ftw f \|_{L^2} \le \|f \|_{\xopt}, \quad r=\| \tau\|_{L^\infty},
	\end{equation}
	 for every $f \in \xopt \cap C(\rd)$ and $\tau \in L^\infty(\rd;\rd)$, $\omega \in L^\infty(\rd;\mathbb{R})$. \par	 
	 Moreover, there exists $C>0$ such that 
\begin{equation}\label{eq maint grad gen}
	\| \Ftw f - f\|_{L^2} \le C(\| \tau \|_{L^\infty} \| \nabla f \|_{\xopt} + \|\omega \|_{L^\infty}\|f\|_{\xopt}), \quad r=\| \tau\|_{L^\infty},
\end{equation}
for every $\tau \in L^\infty(\rd;\rd)$,  $\omega \in L^\infty(\rd;\mathbb{R})$ and $f \in \xopt$ with $\| \nabla f \|_{\xopt}<\infty$. 
\end{theorem}

With the same arguments of the proofs of Theorems \ref{thm def bound Us}, using the bounds in Theorem \ref{thm gen xopt} whenever appropriate, we obtain the following generalization. 
\begin{theorem} 
	Let $\phi \in L^2(\rd)$ be such that Assumptions \ref{asu riesz}, \ref{asu rev 1} and \ref{asu rev 2} in Section \ref{sec mult} hold. There exists a constant $C>0$ such that
\begin{equation}
	\| \Ftw f - f \|_{L^2}  \, \le \begin{cases}
		C (\|\tau\|_{L^\infty}/s+\|\omega\|_{L^\infty}) \|f \|_{L^2} & (\| \tau \|_{L^\infty}/s \le 1) \\
		C (\|\tau\|_{L^\infty}/s)^{d/2} \|f \|_{L^2} & (\| \tau \|_{L^\infty}/s \ge 1)
	\end{cases},
\end{equation}
for every $s>0$, $f \in U_s$ and $\tau \in L^\infty(\rd;\rd)$,  $\omega \in L^\infty(\rd;\mathbb{R})$. 
\end{theorem}

We remark that for band-limited functions $U_s = \mathrm{PW}_R$ with $s=\pi/R$ and in the relevant case where $R\|\tau \|_{L^\infty} \le 1$ we recover the same bounds proved in \cite{wiat paper} without extra regularity conditions on $\tau$ or $\omega$. Similarly, one could generalize the estimates in Besov spaces of the previous section. 

\section{Sharpness of the estimates}\label{sec sharp pwr} 
We now study the problem of the sharpness of some estimates proved so far, focusing in particular on the case of band-limited functions. 

For $R>0$ consider the space of band-limited functions 
\[ \mathrm{PW}_R \coloneqq \{ f \in L^2(\rd) : \supp \hat{f} \subset [-R,R]^d \}. \]
We already commented in Example \ref{rem examples} that such a space of low-frequency functions can be equivalently designed as a multiresolution space; precisely, we have $\mathrm{PW}_R = U_s$ with $s=\pi / R$ after choosing the normalized low-pass sinc filter $\phi=\phi_0 \otimes \cdots \otimes \phi_0$ ($d$ times), with $\phi_0(t) = \pi^{-1/2}\sin t/t$, $t \in \bR$, which satisfies Assumptions \ref{asu riesz}, \ref{asu rev 1}, \ref{asu rev 2}. 

Theorems \ref{thm bound Us} and \ref{thm def bound Us} above thus cover the case of band-limited approximations. Precisely, \eqref{eq bound Us} now reads 
\begin{equation}\label{eq bound pwr} \| F_\tau f \|_{L^2} \le C (1+R\|\tau\|_{L^\infty})^{d/2}\| f\|_{L^2}, \quad f \in \mathrm{PW}_R, \end{equation} while \eqref{eq def bound Us} becomes
\begin{equation}\label{eq def bound pwr} \| F_\tau f - f \|_{L^2}  \, \le \begin{cases}
		C R\|\tau\|_{L^\infty} \|f \|_{L^2} & (R\| \tau \|_{L^\infty}\le 1) \\
		C (R\|\tau\|_{L^\infty})^{d/2} \|f \|_{L^2} & (R\| \tau \|_{L^\infty} \ge 1)
	\end{cases}, \quad  f \in \mathrm{PW}_R. \end{equation}
We claim that the exponent $d/2$ appearing in the previous estimates is optimal. For what concerns \eqref{eq bound pwr}, it suffices to consider $f_R \in \mathrm{PW}_R$ given by $f_R=R^{d/2} D_R \phi$, so that $\| f_R\|_{L^2} = 1$ and $\widehat{f_R} = (\pi/R)^{d/2} \mathbbm{1}_{[-R,R]^d}$. Now, for $K>0$ set
\[ \tau(x) = \begin{cases}
	x & (|x|\le K) \\ 0 & (|x|>K)
\end{cases}, \] so that $\| \tau \|_{L^\infty} = K$. Then, for $|x|\le K$ we have
\[ F_\tau f_R (x) = f_R(0)= (R/\pi)^{d/2}, \] and thus 
\[
\| F_\tau f_R \|_{L^2} \gtrsim (R \| \tau \|_{L^\infty})^{d/2}.
\]
By the triangle inequality we also deduce
\begin{equation}\label{stima dal basso}
	\| F_\tau f_R - f_R \|_{L^2}  \gtrsim (R \| \tau \|_{L^\infty})^{d/2},\quad  R \| \tau \|_{L^\infty}\gg 1,
\end{equation}which shows the sharpness of the exponent $d/2$ in \eqref{eq def bound pwr} as well.

Concerning the sharpness of the estimate \eqref{eq def bound pwr} in the regime $R\| \tau \|_{L^\infty}\ll 1$ we see that if $f=f_R$ as above and $\tau(x)=(c,0,\ldots,0)\in\rd$ (constant), for $|c|R$ small enough we have 
\[
\|F_\tau f_R-f_R\|^2_{L^2}=\Big(\frac{1}{2R}\Big)^d\int_{[-R,R]^d} |e^{-ic\omega_1}-1|^2\, d\omega\gtrsim R^{-d}\int_{[-R,R]^d}(c\omega_1)^2\, d\omega\gtrsim (cR)^2.
\]

\section{Random deformations}\label{sec random}
We now model the deformation $\tau(x)$ as a measurable random field, i.e.\ $\tau(x)=\tau(x,\omega)$ depends on an additional variable\footnote{In this section we do not consider frequency-modulated deformations, nor we use the notation $\omega$ for the frequency, hence there is not risk of confusion with the notation of previous sections.} $\omega \in \mathcal{U}$, where the sample space $\mathcal{U}$ is equipped with a probability measure $\mathbb{P}$, and the function $\tau(x,\omega)$ is jointly measurable (see for instance \cite[Chapter 3]{gihman} for further details). 

It is easy to realize that the results of the previous sections hold for almost every realization of $\tau(x)$ if, e.g., $\|\tau\|_{L^\infty}<\infty$, which must be intended hereinafter as the essential supremum jointly in $x,\omega$. However, it turns out that some results hold, in fact, in a \textit{maximal} sense\footnote{Actually, we could equivalently reformulate the main estimates of the previous sections as results for the maximal operators $\sup_{|y|\leq r}|f(x-y)|$ and $\sup_{|y|\leq r}|f(x-y)-f(x)|$. However, the above presentation in terms of their linearized versions $F_\tau$ and $F_\tau-I$ seems closer to the spirit of the intended  applications.}. Precisely, an inspection of the proof of the formula \eqref{eq xopt bound} shows that we have
\begin{equation}\label{eq xopt bound max}
		 \| \| F_\tau f \|_{L^\infty(\mathcal{U})} \|_{L^2} \le \|f \|_{\xopt}, \quad r=\| \tau\|_{L^\infty},
	\end{equation}
and similarly \eqref{eq maint grad xopt} becomes
\begin{equation}\label{eq maint grad xopt max}
	\|\|F_\tau f-f \|_{L^\infty(\mathcal{U})}\|_{L^2}\le C \|\tau\|_{L^\infty} \|\nabla f \|_{\xopt}, \quad r=\|\tau\|_{L^\infty}.
\end{equation} 
As a consequence, under the assumptions of Theorem \ref{thm def bound Us} we have, for $f\in U_s$,
	\begin{equation}\label{eq def bound Us max}
		\|\|F_\tau f-f\|_{L^\infty(\mathcal{U})}\|_{L^2}  \, \le \begin{cases}
			 C (\|\tau\|_{L^\infty}/s) \|f \|_{L^2} & (\| \tau \|_{L^\infty}/s \le 1) \\
		C (\|\tau\|_{L^\infty}/s)^{d/2} \|f \|_{L^2} & (\| \tau \|_{L^\infty}/s \ge 1)
		\end{cases},
	\end{equation}
	while arguing as in the proof of Theorem \ref{main teo 0} we get 
	\begin{equation}\label{eq 4 star max}
\|\|F_\tau f-f\|_{L^\infty(\mathcal{U})}\|_{L^2} \leq C(\|\tau\|_{L^\infty}\|f\|_{\dot{B}^{1}_{2,1}} +\|\tau\|_{L^\infty}^{d/2}\|f\|_{\dot{B}^{d/2}_{2,1}}), \qquad d\geq 2,
\end{equation}
and
\begin{equation}\label{eq 5 star max}
\|\|F_\tau f-f\|_{L^\infty(\mathcal{U})}\|_{L^2}\leq C \|\tau\|_{L^\infty}^{1/2}\|f\|_{\dot{B}^{1/2}_{2,1}}, \quad d=1.
\end{equation}
We are now ready to state our result concerning the stability in mean under random deformations. 
\begin{theorem}\label{teo random} Under the assumption \ref{asu riesz}, \ref{asu rev 1} and \ref{asu rev 2} in Section \ref{sec mult}, there exists a constant $C>0$ such that, for every $s>0$ and $f\in U_s$, 
\begin{equation}\label{eq stima random 1}
\mathbb{E} \|F_\tau f - f\|_{L^2}^2\leq C\mathbb{E}[(|\tau|/s)^2+(|\tau|/s)^d]\|f\|^2_{L^2},\quad d\geq 2,
\end{equation}
and 
\begin{equation}\label{eq stima random 2}
\mathbb{E} \|F_\tau f - f\|_{L^2}^2\leq C \mathbb{E}[ \min\{(|\tau|/s)^2,(|\tau|/s)^d\}]|f\|^2_{L^2}, \quad d=1,
\end{equation} for every measurable random function $\tau$ such that the random variables $|\tau(x)|$, $x \in \rd$, are identically distributed and the above moments are finite. 

Moreover, if the spaces $U_{2^j}$, $j\in\mathbb{Z}$, define a multiresolution approximation of $L^2(\rd)$, for the same deformations $\tau(x)$ and every $f\in L^2(\rd)$ with $\|f\|_{\dot{B}^{d/2}_{2,1}}  <\infty$ we have
\begin{equation}\label{eq 6 star max}
\mathbb{E} \|F_\tau f - f\|_{L^2}^2\leq C (\mathbb{E}[|\tau|^2] \|f\|^2_{\dot{B}^{1}_{2,1}} +\mathbb{E}[|\tau|^d]\|f\|^2_{\dot{B}^{d/2}_{2,1}}) \qquad d\geq 2
\end{equation}
and
\begin{equation}\label{eq 7 star max}
\mathbb{E} \|F_\tau f - f\|_{L^2}^2 \leq C\mathbb{E}[|\tau|]\|f\|^2_{\dot{B}^{1/2}_{2,1}} \quad d=1.
\end{equation}
\end{theorem}
For the sake of brevity, we wrote $\mathbb{E}[|\tau|^2]$ in place of $\mathbb{E}[|\tau(x)|^2]$, and similarly for the other moments, since the variables $|\tau(x)|$, $x\in \rd$, are assumed to be identically distributed. However, observe that the field $\tau(x)$ is not assumed to be bounded.
\begin{proof}[Proof of Theorem \ref{teo random}]
Let us prove \eqref{eq stima random 1} and \eqref{eq stima random 2} first. Let us set
\[ \tau_j(x) \coloneqq \begin{cases} \tau(x) & (2^{j-1}<|\tau(x)|\leq 2^j) \\ 0 & (\text{otherwise}) \end{cases}, \quad j \in \bZ. \] Then we can write
\begin{align*}
\|F_\tau f-f\|^2_{L^2}&=\sum_{j\in\mathbb{Z}} \int_{\rd} |F_{\tau_j}f(x)-f(x)|^2 \mathbbm{1}_{\{2^{j-1}<|\tau|\leq 2^j\}}(x)\, dx\\
&\leq \sum_{j\in\mathbb{Z}} \int_{\rd} \mathbbm{1}_{\{2^{j-1}<|\tau|\leq 2^j\}}(x) \| F_{\tau_j}f(x)-f(x) \|^2_{L^\infty(\mathcal{U})} \, dx.
\end{align*}
Taking the expectation and setting $p_j=\mathbb{P}(\{2^{j-1}<|\tau(x)|\leq 2^j\})$ (note that $p_j$ is independent of $x$) we get
\[
\mathbb{E} \|F_\tau f-f\|^2_{L^2}\leq \sum_{j\in\mathbb{Z}} p_j \| \| F_{\tau_j}f-f\|_{L^\infty(\mathcal{U})} \|_{L^2}^2.
\]
We use the estimate \eqref{eq def bound Us max} to bound each term and we obtain 
\[
\mathbb{E} \|F_\tau f-f\|^2_{L^2}\lesssim \lc \sum_{2^j\leq s} p_j (2^j/s)^2 +  \sum_{s<2^j} p_j (2^j/s)^d \rc \|f\|^2_{L^2},\quad f\in U_s.
\]
We now observe that, for every $x\in\rd$,
\[
\sum_{2^j\leq s} p_j (2^j/s)^2 = \sum_{2^j\leq s} \mathbb{E}[(2^j/s)^2\mathbbm{1}_{\{2^{j-1}<|\tau(x)|\leq 2^j\}}] \lesssim \mathbb{E}[ (|\tau(x)|/s)^2\mathbbm{1}_{\{|\tau(x)|/s\leq 1\}}]
\]
and similarly
\begin{align*}
 \sum_{s<2^j} p_j (2^j/s)^d &\lesssim \mathbb{E}[(|\tau(x)|/s)^d\mathbbm{1}_{\{|\tau(x)|/s> 1/2\}}]\\
 &\lesssim \mathbb{E}[ (|\tau(x)|/s)^2\mathbbm{1}_{\{1/2<|\tau(x)|/s\leq 1\}}+(|\tau(x)|/s)^d\mathbbm{1}_{\{|\tau(x)|/s> 1\}}].
\end{align*}
Hence we have proved the estimate 
\[
\mathbb{E} \|F_\tau f-f\|^2_{L^2}\lesssim \mathbb{E}[ (|\tau(x)|/s)^2\mathbbm{1}_{\{|\tau(x)|/s\leq 1\}}+(|\tau(x)|/s)^d\mathbbm{1}_{\{|\tau(x)|/s> 1\}}]\|f\|^2_{L^2},
\]
which gives \eqref{eq stima random 1} and \eqref{eq stima random 2}.

Similar arguments lead to the proof of \eqref{eq 6 star max} and \eqref{eq 7 star max}, now using \eqref{eq 4 star max} and \eqref{eq 5 star max}.
\end{proof}

%\bibliographystyle{plain}
%\bibliography{reference}

\begin{thebibliography}{99}

	\bibitem{aliprantis} Charalambos D. Aliprantis, and Kim Border. \textit{Infinite dimensional Analysis. A Hitchhiker's Guide.} Springer, 2006. 

	\bibitem{balan} Radu Balan, Maneesh Singh and Dongmian Zou. Lipschitz properties for deep convolutional networks. \textit{Contemporary Mathematics} \textbf{706} (2018), 129--151.
    
	\bibitem{bietti0} Alberto Bietti and Julien Mairal. Invariance and stability of deep convolutional representations. In: \textit{Advances in Neural Information Processing Systems (NIPS)}, 2017. 

	\bibitem{bietti} Alberto Bietti and Julien Mairal. Group invariance, stability to deformations, and complexity of deep convolutional representations. \textit{Journal of Machine Learning Research (JMLR)} \textbf{20}(25) (2019):1--49.
    
    \bibitem{bro_gdl} Michael M. Bronstein, Joan Bruna, Taco Cohen, Petar Veličković. \textit{Geometric Deep Learning: Grids, Groups, Graphs, Geodesics, and Gauges}. To appear --- MIT Press, 2025. arXiv:2104.13478.
    
    \bibitem{brunamallat} Joan Bruna and St\'ephane Mallat. Invariant scattering convolution networks. \textit{IEEE Transactions on Pattern Analysis and Machine Intelligence (PAMI)}, \textbf{35}(8) (2013), 1872--1886.

	\bibitem{cn sharp} Elena Cordero and Fabio Nicola. Sharpness of some properties of Wiener amalgam and modulation spaces. \textit{Bull. Aust. Math. Soc.} \textbf{80} (2009), no. 1, 105--116. 

	\bibitem{dan} Piero D'Ancona and Fabio Nicola. Sharp $L^p$ estimates for Schr\"odinger groups. \textit{Rev. Mat. Iberoam.} \textbf{32} (2016), no. 3, 1019--1038.
	
	\bibitem{evans} Lawrence C. Evans. Partial differential equations. Second edition. Graduate Studies in Mathematics, 19. American Mathematical Society, Providence, RI, 2010. 
	
	\bibitem{fei83} Hans G. Feichtinger. Banach convolution algebras of Wiener type. In: \textit{Functions, series, operators, Vol. I, II (Budapest, 1980)}, Colloq. Math. Soc. J\'anos Bolyai, 35, North-Holland, Amsterdam, 1983, 509--524. 
	
	\bibitem{fei81} Hans G. Feichtinger. Banach spaces of distributions of Wiener's type and interpolation. In: \textit{Functional analysis and approximation (Oberwolfach, 1980)}, Internat. Ser. Numer. Math., 60, Birkh\"auser, Basel-Boston, Mass., 1981, 153--165. 
	
	\bibitem{grafakos_mod} Loukas Grafakos. \textit{Modern Fourier analysis}. Third edition. Graduate Texts in Mathematics. Springer, New York, 2014.  
	
	\bibitem{grobook} Karlheinz Gr\"ochenig. \textit{Foundations of time-frequency analysis}. Applied and Numerical Harmonic Analysis. Birkh\"auser Boston, Inc., Boston, MA, 2001. 
	
	\bibitem{gihman} Iosif I. Gihman and Anatolij V. Skorohod. \textit{The theory of stochastic processes I }. Translated from the Russian by Samuel Kotz. Corrected reprint of the first English edition. Grundlehren der Mathematischen Wissenschaften [Fundamental Principles of Mathematical Sciences], 210. Springer-Verlag, Berlin-New York, 1980. 
	
	\bibitem{grohs wiat} Philipp Grohs, Thomas Wiatowski and Helmut B\"olcskei. Deep convolutional neural networks on cartoon functions. In: \textit{2016 IEEE International Symposium on Information Theory (ISIT)}, 2016, 1163--1167.
	
	\bibitem{heil} Christopher Heil. An introduction to weighted Wiener amalgams. In:	\textit{Wavelets and their applications}. Ed.\ by S.\ Thangavelu, M.\ Krishna,	R.\ Radha. Allied Publishers, New Dehli, 2003, 183--216.

	\bibitem{koller} Michael Koller, Johannes Gro\ss mann, Ullrich Monich and Holger Boche. Deformation stability of deep convolutional neural networks on Sobolev spaces. In: \textit{2018 IEEE International Conference on Acoustics, Speech and Signal Processing (ICASSP)}, Calgary, AB, 2018, 6872--6876.
		
	\bibitem{mall cpam} St\'ephane Mallat. Group invariant scattering. \textit{Comm. Pure Appl. Math.} \textbf{65} (2012), no. 10, 1331--1398. 
	
	\bibitem{mallbook} St\'ephane Mallat. \textit{A wavelet tour of signal processing. The sparse way.} Third edition. With contributions from Gabriel Peyr\'e. Elsevier/Academic Press, Amsterdam, 2009.
	
	\bibitem{NT_stab} Fabio Nicola, S. Ivan Trapasso. Stability of the scattering transform for deformations with minimal regularity. \textit{J. Math. Pures Appl.\ (9)} \textbf{180} (2023), 122--150.
	
	\bibitem{pixdef} Aaditya Prakash, Nick Moran, Solomon Garber, Antonella DiLillo and James Storer. Deflecting adversarial attacks with pixel deflection. 2018 IEEE/CVF Conference on Computer Vision and Pattern Recognition, 2018, pp. 8571--8580, doi: 10.1109/CVPR.2018.00894.

    \bibitem{scaman} Kevin Scaman and Aladin Virmaux. Lipschitz regularity of deep neural networks: analysis and efficient estimation. In: \textit{Proceedings of the 32nd International Conference on Neural Information Processing Systems (NIPS 2018)}. 

  	\bibitem{stein} Elias M. Stein. \textit{Singular Integrals and Differentiability Properties of Functions}. Princeton Mathematical Series, No. 30 Princeton University Press, Princeton, N.J. 1970 {\rm xiv}+290 pp.
	
	\bibitem{tao} Terence Tao. Low regularity semi-linear wave equations. \textit{Comm. Partial Differential Equations} \textbf{24} (1999), no. 3-4, 599--629.
	
	\bibitem{wiat paper} Thomas Wiatowski and Helmut B\"olcskei. A mathematical theory of deep convolutional neural networks for feature extraction. \textit{IEEE Trans. Inform. Theory} \textbf{64} (2018), no. 3, 1845--1866.  
	
	\bibitem{wiat old} Thomas Wiatowski and Helmut B\"olcskei. Deep convolutional neural networks based on semi-discrete frames. In: \textit{2015 IEEE International Symposium on Information Theory (ISIT), Hong Kong, China}, 2015, 1212--1216.
	
	\bibitem{zou} Dongmian Zou, Radu Balan and Maneesh Singh. On Lipschitz bounds of general convolutional neural networks. \textit{IEEE Trans. on Info. Theory} \textbf{66}(3) (2020), 1738--1759.
\end{thebibliography}

\section*{Acknowledgements}

The authors wish to express their gratitude to Giovanni S.\ Alberti, Enrico Bibbona and Matteo Santacesaria for fruitful conversations on the topics of the manuscript, as well as for valuable comments on preliminary drafts.

The present research has been partially supported by the MIUR grant Dipartimenti di Eccellenza 2018-2022, CUP: E11G18000350001, DISMA, Politecnico di Torino. 

S. Ivan Trapasso was member of the Machine Learning Genoa (MaLGa) Center, Università di Genova, when this study was performed. This material is based upon work supported by the Air Force Office of Scientific Research under award number FA8655-20-1-7027. 

F. Nicola is a fellow of the Accademia delle Scienze di Torino, and a member of the Societ\`a Italiana di Scienze e Tecnologie Quantistiche (SISTEQ).

The authors are members of the Gruppo Nazionale per l'Analisi Matematica, la Probabilità e le loro Applicazioni (GNAMPA) of the Istituto Nazionale di Alta Matematica (INdAM).

\end{document}